\documentclass[1[leqno,11pt]{amsart}
\usepackage{amssymb,amsmath,amsmath,latexsym,amssymb,amsfonts,amsbsy, amsthm}
\numberwithin{equation}{section}

\allowdisplaybreaks
\let\f=\frac

\let\Om=\Omega

\let\na=\nabla

\def\R{\mathbb R}
\def\T{\mathbb T}

\newcommand{\beq}{\begin{equation}}
\newcommand{\eeq}{\end{equation}}
\newcommand{\ben}{\begin{eqnarray}}
\newcommand{\een}{\end{eqnarray}}
\newcommand{\beno}{\begin{eqnarray*}}
\newcommand{\eeno}{\end{eqnarray*}}
\newcommand{\ba}{\begin{array}}
\newcommand{\ea}{\end{array}}

\newtheorem{theorem}{Theorem}[section]

\newtheorem{lemma}[theorem]{Lemma}

\begin{document}

\title[Global well-posedness of the MHD equations]
{Global well-posedness of the MHD equations via the comparison principle}

\author{Dongyi Wei}
\address{School of Mathematical Sciences, Peking University, 100871, Beijing, P. R. China}
\email{jnwdyi@163.com}

\author{Zhifei Zhang}
\address{School of Mathematical Sciences, Peking University, 100871, Beijing, P. R. China}
\email{zfzhang@math.pku.edu.cn}

\date{\today}%7.25 2016 final version

\begin{abstract}
In this paper, we prove the global well-posedness of the incompressible MHD equations near a homogeneous equilibrium in the domain $\R^k\times\mathbb{T}^{d-k}, d\geq2,k\geq1$ by using the comparison principle and constructing the comparison function.
 \end{abstract}

\maketitle

\section{Introduction}

In this paper, we consider the incompressible magneto-hydrodynamics (MHD) equations in $[0,T)\times \Omega, \Omega\subseteq\R^{d}, d\ge 2$:
\begin{align}\label{eq:MHD}
\left\{
\begin{array}{l}
\partial_t v-\nu\Delta v+v\cdot\nabla v+\nabla p=b\cdot\nabla b,\\
\partial_t b-\mu\Delta b+v\cdot\nabla b=b\cdot\nabla v,\\
\text{div } v=\text{div } b=0,
\end{array}\right.
\end{align}
where $v$ denotes the velocity field and $b$ denotes the magnetic field, and $\nu\ge 0$ is the viscosity coefficient, $\mu\ge 0$ is the resistivity coefficient.  If $\nu=\mu=0$, (\ref{eq:MHD}) is the so called ideal MHD equations; If $\nu>0$ and $b=0$, (\ref{eq:MHD}) is reduced to the Navier-Stokes equations.

It is well-known that the 2-D MHD equations with full viscosities(i.e., $\nu>0$ and $\mu>0$) have global smooth solution. In the absence of resistivity(i.e., $\mu=0$), the global existence of weak solution and strong solution of the MHD equations is still an open question. Recently, Cao and Wu \cite{CW}  proved the global regularity of the 2-D MHD equations with partial dissipation and magnetic diffusion. Motivated by numerical observation \cite{CC}: the energy of the MHD equations is dissipated at a rate independent of the ohmic resistivity, there are a lot of works \cite{AZ, LXZ, Ren1, Ren2, Zhang}  devoted to the global well-posedness of the MHD equations without resistivity in a homogeneous magnetic field $B_0$.

In high temperature plasmas, both $\nu$ and $\mu$ are usually very small.  Thus, it is very interesting to investigate the long-time dynamics of the MHD equations in such case. In this paper, we consider the case of $\mu=\nu$. In terms of  the Els$\ddot{a}$sser variables
$$Z_+=v+b,\quad Z_-=v-b,$$
the MHD equations  (\ref{eq:MHD})  can be written as
\begin{align}\label{eq:MHD-e}
\left\{
\begin{array}{l}
\partial_t Z_++Z_-\cdot\nabla Z_+=\mu\Delta Z_+-\nabla p,\\
\partial_t Z_-+Z_+\cdot\nabla Z_-=\mu\Delta Z_--\nabla p,\\
\text{div} Z_+=\text{div}Z_-=0.
\end{array}\right.
\end{align}
 We introduce the fluctuations
 \beno
 z_+=Z_+-B_0,\quad z_-=Z_-+B_0.
 \eeno
Then the system (\ref{eq:MHD-e}) can be reformulated as
\begin{align}\label{eq:MHD-f}
\left\{
\begin{array}{l}
\partial_t z_++Z_-\cdot\nabla z_+=\mu\Delta z_+-\nabla p,\\
\partial_t z_-+Z_+\cdot\nabla z_-=\mu\Delta z_--\nabla p,\\
\text{div} z_+=\text{div}z_-=0.
\end{array}\right.
\end{align}

Bardos-Sulem-Sulem \cite{BSS} proved the global well-posedness of (\ref{eq:MHD-f}) for $\Om=\R^d$ and $\mu=0$ when the initial data is small in a weighed H\"{o}lder space.  Recently, He-Xu-Yu \cite{HXY} proved the global well-posedness of (\ref{eq:MHD-f})  for any $\mu\ge 0$ and $\Om=\R^3$ by using some ideas from nonlinear stability of Minkowski space-time in general relativity. Cai and Lei \cite{CL} proved simliar result for $\Om=\R^d, d=2,3$ by using Alinhac's ghost weight method.  Wei and Zhang \cite{WZ} dealt with more physical case, which allows $\nu\neq \mu$ and $\Om$ to be a strip. For all these results, the key mechanism leading to the global well-posedness is that  the nonlinear terms $z_-\cdot\na z_+$ and $z_+\cdot\na z_-$ are essentially neglected after a long time, because $z_\pm$  are transported along the opposite direction.

The goal of this paper is twofold: (1) include the domain $\R^k\times \T^{d-k}, k\ge 1$. Previous results required $k\ge 2$ at least; (2) develop an elementary and much simpler method via the comparison principle.\smallskip

Without loss of generality, we take the background magnetic field $B_0=(1,0,\cdots,0)$.  We will identify a function in $\R^k\times \T^{d-k}$ with a function in $\R^d$ periodic in $ d-k$ directions $ e_1,\cdots, e_{d-k}$, where $e_1,\cdots, e_{d-k},B_0$ are orthogonal.

We introduce
\begin{align}\label{rho1}\rho_{\pm}(t,X)^2=\sum_{k=0}^N\int_{\R^d}|\nabla^k z_{\pm}|^2(t,Y)\theta(|X-Y|)dY,
\end{align}
where $\theta(r)$ is a smooth cut-off function so that
$$\theta(r)=\left\{\ba{l}1\quad\text{for}\ |r|\leq 1,\\0\quad\text{for}\ |r|\geq 2,\ea\right.\ \ |\theta'(r)|^2\leq C\theta(r).$$
Let
\begin{align}\label{J1}J_{\pm}(0)=\int_{\R}\sup_{y\in\R^{d-1}}\rho_{\pm}(0,x,y)dx.
\end{align}

Our result is stated as follows.

\begin{theorem}\label{thm:main}\quad
Let $d\ge 2, 1\le k\le d$ and $z_\pm(0)\in H^N(\R^k\times\T^{d-k})$ for some integer $N>\f d2+1$.
There exists $\epsilon_1>0$ so that if $J_\pm(0)\le \epsilon_1$, then there exists a global unique solution $(z_+,z_-)\in C\big([0,+\infty),H^N(\R^k\times\mathbb{T}^{d-k})\big)$ to  the MHD equations (\ref{eq:MHD-f}) satisfying
\beno
\|z_\pm(t)\|_{H^N(\R^k\times\T^{d-k})}\le C\|z_\pm(0)\|_{H^N(\R^k\times\T^{d-k})}\quad \text{for any}\,\, t\in [0,+\infty).
\eeno
\end{theorem}

Let us give some remarks on our result.

\begin{itemize}

\item[1.] We only require the initial data to decay at infinity in $B_0$ direction, which is the key for the global well-posedness in $\R\times\mathbb{T}^{d-1}$.

\item[2.] It is easy to check that $J_\pm(0)\le C\epsilon$ for the initial data considered in \cite{CL} and \cite{HXY}. Indeed, for the data in \cite{CL}, we have $\rho_{\pm}(0,X)\leq C\epsilon(1+|X|)^{-\delta} $. For the data in \cite{HXY},  we have $\rho_{\pm}(0,X)\leq C\epsilon(R^2+|X|^2)^{-\frac{1}{2}}(\ln(R^2+|X|^2))^{-2}. $ Here $\delta>1,\ R\geq 100 $ and $\epsilon>0$ is small.

\item[3.]  At a first glance, there seems no difference between $\R^d$ and $\R^k\times\T^{d-k}$. From the proofs in \cite{CL} and \cite{HXY}, it seems that the case of $d=3$ is easier than the case of $d=2$. However, for a solution $(z_+,z_-)\in C\big([0,+\infty),H^N(\R^2)\big)$,  $\big(z_+(t,x),0,z_-(t,x),0\big)$ is also a solution in $C\big([0,+\infty),H^N(\R^2\times\mathbb{T})\big)$. Thus, the case of $\R^2\times\T$ is not easier than the case of $\R^2$. In this sense, the case of $\R\times\T^{d-1}$ may be harder.

\end{itemize}

Throughout this paper, we denote by $C$ a constant independent
of $t,\mu$, which may be different from line to line.

\section{Local energy inequality}

In this section, we derive the following local energy inequalities
 \begin{align}\label{eq:local}
 \partial_t\rho_{\pm}\mp B_0\cdot\nabla\rho_{\pm}-\mu\triangle\rho_{\pm}\leq CF,\end{align}
 where $F(t,X)$ is given by
 \ben\label{def:F}
F(t,X)=\sum_{k+j\leq N+1,0\leq k,j\leq N}\big\||\nabla^{k}z_+||\nabla^{j}z_-| (t)\big\|_{L^2(B(X,2))}+\sum_{k=0}^N\big\|\nabla^{k+1} p(t)\big\|_{L^2(B(X,2))}.
\een
Here $B(X,r)$ denotes a ball in $\R^d$ with the center at $X$ and radius $r$.

We only prove \eqref{eq:local} for $\rho_{+} $, the case of $\rho_{-} $ is similar.  For any multi-index $a\in \mathbb{N}^d$, using the same notations as in \cite{CL}, one can easily deduce from \eqref{eq:MHD-f} that \begin{align*}
\left\{
\begin{array}{l}
\partial_t \nabla^az_+-\mu \Delta \nabla^az_+-B_0\cdot\nabla \nabla^az_++\dfrac{}{}\sum\limits_{b+c=a}C_a^b(\nabla^bz_-\cdot\nabla\nabla^cz_+)+\nabla\nabla^a p=0,\\
\partial_t \nabla^az_--\mu \Delta \nabla^az_-+B_0\cdot\nabla \nabla^az_-+\dfrac{}{}\sum\limits_{b+c=a}C_a^b(\nabla^bz_+\cdot\nabla\nabla^cz_-)+\nabla\nabla^a p=0,\\
\text{div} \nabla^az_+=\text{div}\nabla^az_-=0.
\end{array}\right.
\end{align*}
Taking inner product of the first equation with $\nabla^az_+ $, we obtain
\begin{align*}
\partial_t |\nabla^az_+|^2&-\mu \Delta |\nabla^az_+|^2+2\mu|\nabla\nabla^az_+|^2-B_0\cdot\nabla| \nabla^az_+|^2\nonumber\\&+2\sum\limits_{b+c=a}C_a^b(\nabla^bz_-\cdot\nabla\nabla^cz_+\nabla^az_+)+2\nabla^az_+\cdot\nabla\nabla^a p=0,
\end{align*}
from which, we deduce that
\begin{align}\label{eq:MHD-d2}
( \partial_t\rho_{+}^2-\mu\triangle\rho_{+}^2- B_0\cdot\nabla\rho_{+}^2)(t,X)+2\mu\int_{\R^d}\sum_{|a|\leq N}|\nabla\nabla^az_+|^2(t,Y)\theta(|X-Y|)dY\nonumber\\ +\dfrac{}{}2\int_{\R^d}\sum_{|a|\leq N}\sum\limits_{b+c=a}C_a^b(\nabla^bz_-\cdot\nabla\nabla^cz_+\nabla^az_+)(t,Y)\theta(|X-Y|)dY\nonumber\\+2\int_{\R^d}\sum_{|a|\leq N}\nabla^az_+\cdot\nabla\nabla^a p(t,Y)\theta(|X-Y|)dY=0.
\end{align}
Let $\rho_{+}^{(\varepsilon)}=(\rho_{+}^2+\varepsilon)^{\frac{1}{2}} $ for $ \varepsilon>0.$  We have
\begin{align*}&2(\rho_{+}^{(\varepsilon)}\nabla\rho_{+}^{(\varepsilon)})(t,X)=(\nabla\rho_{+}^2)(t,X)=2\int_{\R^d}\sum_{|a|\leq N}(\nabla\nabla^az_{+}\cdot\nabla^a z_{+})^2(t,Y)\theta(|X-Y|)dY\\ &\leq 2\left(\int_{\R^d}\sum_{|a|\leq N}|\nabla\nabla^az_{+}|^2(t,Y)\theta(|X-Y|)dY\int_{\R^d}\sum_{|a|\leq N}|\nabla^az_{+}|^2(t,Y)\theta(|X-Y|)dY\right)^{\frac{1}{2}}\\&=2\left(\int_{\R^d}\sum_{|a|\leq N}|\nabla\nabla^az_{+}|^2(t,Y)\theta(|X-Y|)dY\right)^{\frac{1}{2}}\rho_{+}(t,X),\end{align*}
which implies that
\begin{align*}
\int_{\R^d}\sum_{|a|\leq N}|\nabla\nabla^az_{+}|^2(t,Y)\theta(|X-Y|)dY\geq |\nabla\rho_{+}^{(\varepsilon)}(t,X)|^2.
\end{align*}
If $|a|\leq N,\ b+c=a,\ c<a$, then $|b|,|c|+1\leq N$ and
 \begin{align*}
 &-\int_{\R^d}(\nabla^bz_-\cdot\nabla\nabla^cz_+\nabla^az_+)(t,Y)\theta(|X-Y|)dY\\& \leq \left(\int_{B(X,2)}|\nabla^{|b|}z_-|^2|\nabla^{|c|+1}z_+|^2(t,Y)dY\int_{\R^d}|\nabla^az_{+}|^2(t,Y)\theta(|X-Y|)dY\right)^{\frac{1}{2}}\\ &\leq F(t,X)\rho_+(t,X).
 \end{align*}
If $|a|\leq N,\ b+c=a,\ c=a$, then $b=(0,\cdots,0)$ and
\begin{align*}
&-2\int_{\R^d}(\nabla^bz_-\cdot\nabla\nabla^cz_+\nabla^az_+)(t,Y)\theta(|X-Y|)dY\\
&=-\int_{\R^d}(z_-\cdot\nabla|\nabla^az_+|^2)(t,Y)\theta(|X-Y|)dY\\ &=-\int_{\R^d}z_-\cdot\nabla\theta(|X-Y|)|\nabla^az_+|^2(t,Y)dY\\ &\leq C\left(\int_{B(X,2)}|z_-|^2|\nabla^{|a|}z_+|^2(t,Y)dY\int_{\R^d}|\nabla^az_{+}|^2(t,Y)\theta(|X-Y|)dY\right)^{\frac{1}{2}}\\
&\leq CF(t,X)\rho_+(t,X) .\end{align*}
If $|a|\leq N,$ then
\begin{align*}
&-\int_{\R^d}\nabla^az_+\cdot\nabla\nabla^a p(t,Y)\theta(|X-Y|)dY\\
&\leq \left(\int_{B(X,2)}|\nabla^{|a|+1}p|^2(t,Y)dY\int_{\R^d}|\nabla^az_{+}|^2(t,Y)\theta(|X-Y|)dY\right)^{\frac{1}{2}}\\
&\leq F(t,X)\rho_+(t,X).
\end{align*}
Summing up, we obtain
\begin{align*}
( \partial_t\rho_{+}^2-\mu\triangle\rho_{+}^2- B_0\cdot\nabla\rho_{+}^2)(t,X)+2\mu|\nabla\rho_{+}^{(\varepsilon)}(t,X)|^2
\leq CF(t,X)\rho_+(t,X),
\end{align*}
which gives
\begin{align*}
 (\partial_t\rho_{+}^{(\varepsilon)}-\mu\triangle\rho_{+}^{(\varepsilon)}- B_0\cdot\nabla\rho_{+}^{(\varepsilon)})(t,X)\leq CF(t,X).
\end{align*}
Now, the local energy inequality \eqref{eq:local} follows by letting $\varepsilon\to0 $.\smallskip

Let us conclude this section by the following estimate for $F(t,X)$:
\begin{align}\label{eq:F-est}
F(t,X)\leq C\int_{\R^d}\frac{\rho_{+}(t,Y)\rho_{-}(t,Y)}{1+|X-Y|^{d+1}}dY.
\end{align}
We need the following fact.

\begin{lemma}\label{lem:local}
It holds that
\beno
\|f\|_{L^2(B(X,2))}\leq C\int_{B(X,3)}\|f\|_{L^2(B(Y,\frac{1}{2}))}dY.
\eeno
\end{lemma}

\begin{proof}
By Fubini theorem, we have
\begin{align*}
\int_{B(X,\frac{1}{2})}\|f\|_{L^2(B(X,2)\cap B(Y,3))}^2dY&=\int_{B(X,3)}\|f\|_{L^2(B(X,2)\cap B(Y,\frac{1}{2}))}^2dY\\
&\leq\int_{B(X,3)}\|f\|_{L^2(B(X,2))}\|f\|_{L^2(B(Y,\frac{1}{2}))}dY.
\end{align*}
As $ B(X,2)\cap B(Y,3)=B(X,2)$ for $Y\in B(X,\frac{1}{2})$, we infer that\begin{align*}
\frac{\omega_d}{2^d}\|f\|_{L^2(B(X,2))}^2\leq \|f\|_{L^2(B(X,2))}\int_{B(X,3)}\|f\|_{L^2(B(Y,\frac{1}{2}))}dY,
\end{align*}
where $ \omega_d$ is the volume of the unit ball in $\R^d$, and this gives the result.
\end{proof}

Now let us prove \eqref{eq:F-est}.
By Sobolev embedding, we have
$$\|\nabla^{k}z_{\pm}(t)\|_{L^{p_k}(B(X,1))}\leq C\|z_{\pm}(t)\|_{H^{N}(B(X,1))}\leq C\rho_{\pm}(t,X),$$
 where $ \dfrac{1}{p_k}=\dfrac{k-1}{2(N-1)}$ for $1\leq k\leq N$ and $ \dfrac{1}{p_k}=0$ for $k=0.$ Thus, for $k+j\leq N+1,0\leq k,j\leq N$, we have $\dfrac{1}{p_k}+ \dfrac{1}{p_j}\leq \frac{1}{2}$ and
 \begin{align*}
 \||\nabla^{k}z_{+}||\nabla^{j}z_{-}|(t)\|_{L^{2}(B(X,1))}&\leq C\|\nabla^{k}z_{+}(t)\|_{L^{p_k}(B(X,1))}\|\nabla^{j}z_{-}(t)\|_{L^{p_j}(B(X,1))}\\
 &\leq C\rho_{+}(t,X)\rho_{-}(t,X).
 \end{align*}
 Due to $\text{div}z_{\pm}=0,$ we infer from the first equation of (\ref{eq:MHD-f}) that
 $$-\triangle p=\partial_i(z^j_+\partial_jz^i_-)=\partial_iz^j_+\partial_jz^i_-=\partial_i\partial_j(z^j_+z^i_-).$$
 Then by the interior elliptic estimates, we get
 \begin{align*} &\|\nabla p(t)\|_{H^N(B(X,\frac{1}{2}))}\leq C\|\nabla p(t)\|_{L^{\infty}(B(X,1))}+C\|\triangle p(t)\|_{H^{N-1}(B(X,1))}\\
 &\leq C\|\nabla p(t)\|_{L^{\infty}(B(X,1))}+C\sum_{k+j\leq N-1,\ k,j\geq 0}\big\||\nabla^{k+1}z_{+}||\nabla^{j+1}z_{-}|(t)\big\|_{L^{2}(B(X,1))}\\
 &\leq C\|\nabla p(t)\|_{L^{\infty}(B(X,1))}+C\rho_{+}(t,X)\rho_{-}(t,X),\end{align*}
 from which and Lemma \ref{lem:local}, we infer that
 \begin{align}\label{p1}
 F(t,X) \leq C\|\nabla p(t)\|_{L^{\infty}(B(X,4))}+C\int_{B(X,3)}\rho_{+}(t,Y)\rho_{-}(t,Y)dY.
 \end{align}
 It remains to estimate $\|\nabla p(t)\|_{L^{\infty}(B(X,4))}$. For this, we use the following representation formula of the pressure $p(t,X)$:
 \begin{align*}
-\nabla p(t,X)= &\int_{\R^d}\nabla N(X-Y)\theta(|X-Y|)\triangle p(t,Y)dY\nonumber\\
&+\int_{\R^{d}}\partial_i\partial_j\Big(\nabla N(X-Y)(1-\theta(|X-Y|))\Big)
(z_+^jz_-^i)(t,Y)dY,
\end{align*}
where $N(X)$ is the Newton potential.
Thanks to $|\nabla^k N(X)|\leq C|X|^{2-d-k}$ and Sobolev embedding, we obtain
\begin{align*}
 |\nabla p(t,X)|&\leq C\int_{B(X,2)}\frac{dY}{|X-Y|^{d-1}}\|\triangle p(t)\|_{L^{\infty}(B(X,2))} + C\int_{\R^{d}}\frac{|z_+||z_-|(t,Y)dY}{1+|X-Y|^{d+1}}\\ &\leq C\|\triangle p(t)\|_{H^{N-1}(B(X,2))} + C\int_{\R^{d}}\frac{\rho_{+}(t,Y)\rho_{-}(t,Y))dY}{1+|X-Y|^{d+1}}.
 \end{align*}
 Notice that $\|\triangle p(t)\|_{H^{N-1}(B(Y,\frac{1}{2}))}\leq C\rho_{+}(t,Y)\rho_{-}(t,Y) $, which along with Lemma \ref{lem:local} gives
\begin{align*}
\|\triangle p(t)\|_{H^{N-1}(B(X,2))} \leq C\int_{B(X,3)}\rho_{+}(t,Y)\rho_{-}(t,Y)dY.
\end{align*}
This shows  that
\begin{align*}
|\nabla p(t,X)|\leq  C\int_{\R^{d}}\frac{\rho_{+}(t,Y)\rho_{-}(t,Y))dY}{1+|X-Y|^{d+1}}.
\end{align*}
Thanks to $ \dfrac{1}{1+|X'-Y|^{d+1}}\leq \dfrac{C}{1+|X-Y|^{d+1}}$ for $X'\in B(X,4)$, we have
\begin{align*}
 \|\nabla p(t)\|_{L^{\infty}(B(X,4))}\leq  C\int_{\R^{d}}\frac{\rho_{+}(t,Y)\rho_{-}(t,Y))dY}{1+|X-Y|^{d+1}}.
 \end{align*}
Inserting this into \eqref{p1},  we arrive at \eqref{eq:F-est}.

\section{Comparison principle}

It follows from \eqref{eq:local} and \eqref{eq:F-est} that
 \begin{align}\label{eq:local-2}
 \partial_t\rho_{\pm}\mp B_0\cdot\nabla\rho_{\pm}-\mu\triangle\rho_{\pm}\leq C_1(\rho_+\rho_-)*N_1,
 \end{align}
 where $N_1(X)=(1+|X|^{d+1})^{-1}$.

 To control $\rho_\pm$, we  establish the following comparison principle.

\begin{lemma}\label{lem:com}
Let $0\leq\rho_{\pm}^1 \in L^{\infty}\cap C^0([0,T)\times\R^d )$ satisfy
 \begin{align}\label{eq:local-com}
 \partial_t\rho_{\pm}^1\mp B_0\cdot\nabla\rho_{\pm}^1-\mu\triangle\rho_{\pm}^1\geq C_1(\rho_+^1\rho_-^1)*N_1.
 \end{align}
If $\rho_{\pm}(0)\leq\rho_{\pm}^1(0) $ in $\R^d$,  then $\rho_{\pm}\leq\rho_{\pm}^1 $ in $[0,T)\times\R^d$.
\end{lemma}

 \begin{proof}
 It follows from \eqref{eq:local-2} and \eqref{eq:local-com} that
 \begin{align*}
 \partial_t(\rho_{\pm}-\rho_{\pm}^1)\mp B_0\cdot\nabla(\rho_{\pm}-\rho_{\pm}^1)-\mu\triangle(\rho_{\pm}-\rho_{\pm}^1)\leq C_1(\rho_+\rho_--\rho_+^1\rho_-^1)*N_1.
 \end{align*}
 Let $f^{+}=\max(f,0)$.  By maximum principle, we deduce that for $t\in [0,T)$,
 \begin{align*}
 \|(\rho_{\pm}-\rho_{\pm}^1)^{+}(t)\|_{L^{\infty}(\R^d)}&\leq \int_0^tC_1\left\|\left( (\rho_+\rho_--\rho_+^1\rho_-^1)(s)*N_1\right)^+\right\|_{L^{\infty}(\R^d)}ds\\ &\leq \int_0^tC_1\left\| (\rho_+\rho_--\rho_+^1\rho_-^1)^+(s)*N_1\right\|_{L^{\infty}(\R^d)}ds\\ &\leq C\int_0^t\left\| (\rho_+\rho_--\rho_+^1\rho_-^1)^+(s)\right\|_{L^{\infty}(\R^d)}ds\\ &\leq CM\int_0^t\left(\left\| (\rho_+-\rho_+^1)^+(s)\right\|_{L^{\infty}(\R^d)}+\left\| (\rho_--\rho_-^1)^+(s)\right\|_{L^{\infty}(\R^d)}\right)ds,
 \end{align*}
where $M=\|\rho_++\rho_-^1\|_{L^{\infty}([0,T)\times\R^d)}<+\infty$. This implies that $ \|(\rho_{\pm}-\rho_{\pm}^1)^{+}(t)\|_{L^{\infty}(\R^d)}=0,$ hence, $\rho_{\pm}(t)\leq \rho_{\pm}^1(t)$ for $t\in [0,T)$.
\end{proof}

To construct comparison functions $\rho_\pm^1$,  we need the following observation.

\begin{lemma}\label{lem:N1}
There exists an absolute constant $C_0>1$ so that
\beno
C_0^{-1}N_1\leq N_1*N_1\leq C_0 N_1,\quad \rho_{\pm}(0)\leq C_0\rho_{\pm}(0)*N_1.
\eeno
Here $\|N_1\|_{L^1(\R^d)}\le C_0$.
\end{lemma}
\begin{proof}
As $N_1(X)\leq CN_1(Y)$ for $|X-Y|\leq1$, we have
\begin{align*}
N_1*N_1(X)\geq\int_{B(X,1)}N_1(Y)N_1(X-Y)dY\geq C^{-1}N_1(X).\end{align*}
Using the fact that
$$
\min(N_1(Y),N_1(X-Y))\leq CN_1(X),
$$
we infer that
\begin{align*}
N_1*N_1(X)=&\int_{\R^d}N_1(Y)N_1(X-Y)dY\\= &\int_{\R^d}\min(N_1(Y),N_1(X-Y))\max(N_1(Y),N_1(X-Y))dY\\ \leq & \int_{\R^d}CN_1(X)(N_1(Y)+N_1(X-Y))dY=2CN_1(X)\|N_1\|_{L^1(\R^d)},\end{align*} which gives the first inequality.

Using $\|z_{\pm} (0)\|_{H^{N}(B(Y,\frac{1}{2}))} \leq C\rho_{\pm}(0,Y) $ and Lemma \ref{lem:local}, we deduce that
\begin{align*}
\rho_{\pm}(0,X)\leq\|z_{\pm} (0)\|_{H^{N}(B(X,2))}\leq C\int_{B(X,3)}\rho_{\pm}(0,Y)dY\leq C\rho_{\pm}(0)*N_1(X),
\end{align*}
which gives the second inequality.
\end{proof}

Now let us construct the comparison function.

\begin{lemma}\label{lem:com fun}
Let $\rho_{\pm}^0 \in L^1\cap C(\R),\ \rho_{\pm}^{00}\in L^2\cap L^{\infty}\cap C(\R^k\times \T^{d-k})$. Assume that
\beno
&&0\leq\rho_{\pm}^{00}(x,y)\leq\rho_{\pm}^0(x)\quad\text{for}\,\, x\in\R,\,y\in\R^{d-1},\\
&&C_0^{-1}\rho_{\pm}^{00}\leq \rho_{\pm}^{00}*N_1\leq C_0\rho_{\pm}^{00}.
\eeno
Then there exists $\epsilon_0>0$ depending only on $C_0,C_1$ such that if $\|\rho_{\pm}^0\|_{L^1(\R)}<\epsilon_0 $, then there exists $0\leq \rho_{\pm}^1 \in L^{\infty}\cap C\big([0,+\infty)\times\R^d\big) $ satisfying \eqref{eq:local-com} and $\rho_{\pm}^{00}\leq\rho_{\pm}^1(0)$. Moreover,
\beno
 \|\rho_{\pm}^1(t)\|_{L^2(\R^k\times \T^{d-k})}\leq C\|\rho_{\pm}^{00}\|_{L^2(\R^k\times \T^{d-k})}.
 \eeno
\end{lemma}

 \begin{proof}
 {\bf Step 1.} Construction of the data

 For $\mu=0$, we set $g_{\pm}^{00}=\rho_{\pm}^{00},\ g_{\pm}^{0}=\rho_{\pm}^{0}$. For $\mu>0$, we set
 \begin{align*}
 g_{\pm}^{00}(X)=\frac{1}{2\mu}\int_0^{+\infty}e^{-\frac{y}{2\mu}}\rho_{\pm}^{00}(X\mp B_0y)dy,\quad g_{\pm}^{0}(x)=\frac{1}{2\mu}\int_0^{+\infty}e^{-\frac{y}{2\mu}}\rho_{\pm}^{0}(x\mp y)dy.
 \end{align*}
 It is easy to check that
 \beno
 &&g_{\pm}^{0}\pm 2\mu \partial_xg_{\pm}^{0}=\rho_{\pm}^0,\quad g_{\pm}^{00}\pm 2\mu B_0\cdot\nabla g_{\pm}^{00}=\rho_{\pm}^{00},\\
 &&\|g_{\pm}^0\|_{L^1(\R)}=\|\rho_{\pm}^0\|_{L^1(\R)},\quad \|g_{\pm}^{00}\|_{L^{\infty}(\R^d)}\leq\|\rho_{\pm}^{00}\|_{L^{\infty}(\R^d)},
 \eeno
 and $0\leq g_{\pm}^{00}(x,y)\leq g_{\pm}^0(x)$ for $x\in\R,y\in\R^{d-1}.$

 Let
 \begin{align*}
 \ h_{\pm}^{0}(x)=\frac{1}{2\epsilon_0}\int_0^{+\infty}(\rho_{\pm}^{0}+g_{\pm}^0)(x\mp y)dy.
 \end{align*}
 Then we have
 \beno
 0\leq  h_{\pm}^{0}<1,\quad \pm\partial_xh_{\pm}^{0}=(\rho_{\pm}^{0}+g_{\pm}^0)/(2\epsilon_0).
 \eeno

 {\bf Step 2.} Construction of comparison function

 Let $(\rho_{\pm}^{01},g_{\pm}^{01})(t,X)$ be the solution to
 \begin{align*}
 \partial_tf\mp B_0\cdot\nabla f-\mu\triangle f=0,\quad  f(0,X)=(\rho_{\pm}^{00},g_{\pm}^{00})(X),
 \end{align*}
 and $(\rho_{\pm}^{11},g_{\pm}^{1},h_{\pm}^{1})(t,x)$ be the solution to  \begin{align*}
 \partial_tf\mp \partial_x f-\mu\partial_x^2 f=0,\quad f(0,x)=(\rho_{\pm}^{0},g_{\pm}^{0},h_{\pm}^{0})(x).
 \end{align*}
 Thanks to the construction of the data, we find that
 \beno
 &&0\leq  h_{\pm}^{1}<1,\quad \pm\partial_xh_{\pm}^{1}=(\rho_{\pm}^{11}+g_{\pm}^1)/(2\epsilon_0),\\
 &&g_{\pm}^{01}\pm 2\mu B_0\cdot\nabla g_{\pm}^{01}=\rho_{\pm}^{01},\\
 &&0\leq \rho_{\pm}^{01}(t,x,y)\leq \rho_{\pm}^{11}(t,x),\quad 0\leq g_{\pm}^{01}(t,x,y)\leq g_{\pm}^1(t,x).
 \eeno

 Now we take $\rho_{\pm}^{1}(t)=C_0\rho_{\pm}^{10}(t)*N_1,$
 where
 \beno
  \rho_{\pm}^{10}(t,X)=\rho_{\pm}^{01}(t,X)+g_{\pm}^{01}(t,X)h_{\mp}^1(t,x).
  \eeno

{\bf Step 3.} Verification of  the conditions

By our construction, it is easy to check that
\begin{align*}
 &\partial_t\rho_{\pm}^{10}\mp B_0\cdot\nabla \rho_{\pm}^{10}-\mu\triangle \rho_{\pm}^{10}=\big(\partial_t\rho_{\pm}^{00}\mp B_0\cdot\nabla \rho_{\pm}^{00}-\mu\triangle \rho_{\pm}^{00}\big)\\
 &\quad+(\partial_tg_{\pm}^{01}\mp B_0\cdot\nabla g_{\pm}^{01}-\mu\triangle g_{\pm}^{01})h_{\mp}^{1}+g_{\pm}^{01}\big(\partial_th_{\mp}^{1}\mp \partial_x h_{\mp}^{1}-\mu\partial_x^2 h_{\mp}^{1}\big)\\
 &\quad-2\mu\ B_0\cdot\nabla g_{\pm}^{01}\partial_x h_{\mp}^{1}\\
 &=0+0\mp 2g_{\pm}^{01}\partial_x h_{\mp}^{1}-2\mu\ B_0\cdot\nabla g_{\pm}^{01}\partial_x h_{\mp}^{1}\\
 &=\mp (2g_{\pm}^{01}\pm2\mu\ B_0\cdot\nabla g_{\pm}^{01})\partial_x h_{\mp}^{1}\\
 &=(g_{\pm}^{01}+ \rho_{\pm}^{01})(\rho_{\mp}^{11}+g_{\mp}^1)/(2\epsilon_0)\geq (g_{\pm}^{01}+ \rho_{\pm}^{01})(\rho_{\mp}^{01}+g_{\mp}^{01})/(2\epsilon_0),
 \end{align*}
 which implies that
 \begin{align}
 \partial_t\rho_{\pm}^{1}\mp B_0\cdot\nabla \rho_{\pm}^{1}-\mu\triangle \rho_{\pm}^{1}\geq C_0/(2\epsilon_0)\big((g_{+}^{01}+ \rho_{+}^{01})(\rho_{-}^{01}+g_{-}^{01})\big) *{N_1}.\label{eq:rho-1}
 \end{align}
 As $0\leq  h_{\pm}^{1}<1$, we have
 \beno
 \rho_{\pm}^{1}(t)\leq C_0\rho_{\pm}^{01}(t)*N_1+C_0g_{\pm}^{01}(t)*N_1.
 \eeno
 Since $\rho_{\pm}^{01}(t) $ and $\rho_{\pm}^{01}(t)*N_1 $ satisfy $ \partial_tf\mp B_0\cdot\nabla f-\mu\triangle f=0$ and $\rho_{\pm}^{01}(0)*N_1=\rho_{\pm}^{00}*N_1\leq C_0\rho_{\pm}^{00}=C_0\rho_{\pm}^{01}(0)$, we conclude that $\rho_{\pm}^{01}(t)*N_1\leq C_0\rho_{\pm}^{01}(t).$
 Thanks to $g_{\pm}^{00}*N_1\leq C_0g_{\pm}^{00}$, we similarly have $g_{\pm}^{01}(t)*N_1\leq C_0g_{\pm}^{01}(t)$. Thus,
 \beno
 0\leq \rho_{\pm}^{1}(t)\leq C_0^2\big(\rho_{\pm}^{01}(t)+g_{\pm}^{01}(t)\big),
 \eeno
 which along with \eqref{eq:rho-1} gives
 \begin{align*}
 \partial_t\rho_{\pm}^{1}\mp B_0\cdot\nabla \rho_{\pm}^{1}-\mu\triangle \rho_{\pm}^{1}\geq ( \rho_{+}^{1}\rho_{-}^{1})(t) *{N_1/(2\epsilon_0 C_0^3)}.
 \end{align*}
 Taking $\epsilon_0=1/(2C_0^3C_1) $, we find that  $0\leq \rho_{\pm}^1\in L^{\infty}\cap C\big([0,+\infty)\times\R^d\big)$ satisfies \eqref{eq:local-com}. As $\rho_{\pm}^{10}(0)\geq \rho_{\pm}^{01}(0)=\rho_{\pm}^{00},$ we have $\rho_{\pm}^{1}(0)=C_0\rho_{\pm}^{10}(0)*N_1\geq C_0\rho_{\pm}^{00}*N_1\geq \rho_{\pm}^{00}.$

By standard energy estimate, we can deduce that
\beno
&&\|\rho_{\pm}^{01}(t)\|_{L^2(\R^k\times \T^{d-k})}\leq \|\rho_{\pm}^{01}(0)\|_{L^2(\R^k\times \T^{d-k})}=\|\rho_{\pm}^{00}\|_{L^2(\R^k\times \T^{d-k})}, \\
&&\|g_{\pm}^{01}(t)\|_{L^2(\R^k\times \T^{d-k})}\leq \|g_{\pm}^{00}\|_{L^2(\R^k\times \T^{d-k})}.
\eeno
For $\mu=0$, $\|g_{\pm}^{00}\|_{L^2(\R^k\times \mathbb{T}^{d-k})}=\|\rho_{\pm}^{00}\|_{L^2(\R^k\times \mathbb{T}^{d-k})}$, and for $\mu>0$, \begin{align*}
 \|g_{\pm}^{00}\|_{L^2(\R^k\times \T^{d-k})}\leq& \frac{1}{2\mu}\int_0^{+\infty}e^{-\frac{y}{2\mu}}\|\rho_{\pm}^{00}(\cdot\mp B_0y)\|_{L^2(\R^k\times \T^{d-k})}dy\\=&\frac{1}{2\mu}\int_0^{+\infty}e^{-\frac{y}{2\mu}}\|\rho_{\pm}^{00}\|_{L^2(\R^k\times \T^{d-k})}dy=\|\rho_{\pm}^{00}\|_{L^2(\R^k\times \T^{d-k})}.
 \end{align*}
 Thus, we obtain
 \begin{align*}
 \|\rho_{\pm}^{1}(t)\|_{L^2(\R^k\times \T^{d-k})}\leq& C_0^2\big(\|\rho_{\pm}^{01}(t)\|_{L^2(\R^k\times \T^{d-k})}+\|g_{\pm}^{01}(t)\|_{L^2(\R^k\times \T^{d-k})}\big)\\
 \leq& C_0^2\big(\|\rho_{\pm}^{00}\|_{L^2(\R^k\times \mathbb{T}^{d-k})}+\|g_{\pm}^{00}\|_{L^2(\R^k\times \mathbb{T}^{d-k})}\big)\\
 \leq& 2C_0^2\|\rho_{\pm}^{00}\|_{L^2(\R^k\times \mathbb{T}^{d-k})}.\end{align*}
 This completes the proof.
 \end{proof}

\section{Proof of Theorem \ref{thm:main}}

 By the local well-posedness result, there exists a unique solution $z_{\pm}\in C\big([0,T^*),H^N(\R^k\times\T^{d-k})\big)$ to the MHD equations \eqref{eq:MHD-f}, where $T^*$ is the maximal existence time of the solution.

Let $\rho_{\pm}^{00}=C_0\rho_{\pm}(0)*N_1 $. Then Lemma \ref{lem:N1} ensures that
\beno
\rho_{\pm}(0)\leq \rho_{\pm}^{00},\quad C_0^{-1}\rho_{\pm}^{00}\leq\rho_{\pm}^{00}*N_1\leq C_0\rho_{\pm}^{00}.
\eeno
Let
\begin{align*}
 \rho_{\pm}^{0}(x)=C_0\int_{\R}\int_{\R^{d-1}}\rho_{\pm}^*(x-x')N_1(x',y')dy'dx',
 \end{align*}
where  $\rho_{\pm}^*(x)=\sup\limits_{y\in\R^{d-1}}\rho_{\pm}(0,x,y)$, thus  $J_{\pm}(0)=\|\rho_{\pm}^*\|_{L^1(\R)}.$ Thanks to
\begin{align*}
 \rho_{\pm}^{00}(x,y)=C_0\int_{\R}\int_{\R^{d-1}}\rho_{\pm}(0,x-x',y-y')N_1(x',y')dy'dx',
 \end{align*}
we find that  $\rho_{\pm}^{00}(x,y)\leq \rho_{\pm}^{0}(x)$ and
\begin{align*}
 \|\rho_{\pm}^{0}\|_{L^1(\R)}=&C_0\int_{\R}\int_{\R}\int_{\R^{d-1}}\rho_{\pm}^*(x-x')N_1(x',y')dy'dx'dx\\ =&C_0\|\rho_{\pm}^*\|_{L^1(\R)}\int_{\R}\int_{\R^{d-1}}N_1(x',y')dy'dx'\\
 =&C_0J_{\pm}(0)\|N_1\|_{L^1(\R^d)}\le C_0^2\epsilon_1<\epsilon_0, \end{align*}
 if $J_{\pm}(0)\leq\epsilon_1=\epsilon_0/2C_0^2 $.

 Now, Lemma \ref{lem:com fun} ensures that there exists $0\leq \rho_{\pm}^1 \in L^{\infty}\cap C\big([0,+\infty)\times\R^d\big) $, which satisfies  \eqref{eq:local-com}, and $\rho_{\pm}(0)\leq\rho_{\pm}^{00}\leq\rho_{\pm}^1(0) $ in $\R^d$,  and
 \beno
 \|\rho_{\pm}^1(t)\|_{L^2(\R^k\times \T^{d-k})}\leq C\|\rho_{\pm}^{00}\|_{L^2(\R^k\times \T^{d-k})}\leq C\|\rho_{\pm}(0)\|_{L^2(\R^k\times \T^{d-k})}\leq C\|z_{\pm}(0)\|_{H^N(\R^k\times \T^{d-k})}.
 \eeno
Then we infer from Lemma \ref{lem:com} that
 $ 0\leq \rho_{\pm}\leq \rho_{\pm}^1$ in $ [0,T)\times\R^d$ for $0<T<T^*$.
 Hence,
 \beno
 \|z_{\pm}(t)\|_{H^N(\R^k\times \T^{d-k})}\leq C\|\rho_{\pm}(t)\|_{L^2(\R^k\times\T^{d-k})}\leq C\|\rho_{\pm}^1(t)\|_{L^2(\R^k\times \T^{d-k})}\leq  C\|z_{\pm}(0)\|_{H^N(\R^k\times \T^{d-k})}
 \eeno
 for any $t\in [0,T)$,  which implies $T^*=+\infty.$

 \section{Decay estimates}

 In this section, we provide some decay estimates of the solution in time when $\mu>0$.
 Here we use the notations in section 3 and section 4.\smallskip

 Using the fact that $(\rho_{\pm}^{01},g_{\pm}^{01})(t,X)=e^{\mu t\triangle}(\rho_{\pm}^{00},g_{\pm}^{00})(X-B_0 t)$, we have
 \begin{align*}
g_{\pm}^{01}(t,X)=\frac{1}{2\mu}\int_0^{+\infty}e^{-\frac{y}{2\mu}}\rho_{\pm}^{01}(t,X\mp B_0y)dy,
 \end{align*}
 which gives
 \beno
  \|g_{\pm}^{01}(t)\|_{L^p(\R^k\times \mathbb{T}^{d-k})}\leq \|\rho_{\pm}^{01}(t)\|_{L^p(\R^k\times \mathbb{T}^{d-k})}=\|e^{\mu t\triangle}\rho_{\pm}^{00}\|_{L^p(\R^k\times \mathbb{T}^{d-k})}.
 \eeno
 Thus, we deduce that for $2\leq p\leq \infty$,
 \beno
 \|\rho_{\pm}^{1}(t)\|_{L^p(\R^k\times \mathbb{T}^{d-k})}\leq C\|e^{\mu t\triangle}\rho_{\pm}^{00}\|_{L^p(\R^k\times \mathbb{T}^{d-k})}.
 \eeno
which along with the fact that $e^{\mu t\triangle}\rho_{\pm}^{00}=C_0e^{\mu t\triangle}\rho_{\pm}(0)*N_1$, we infer that
\begin{align*}
\|\rho_{\pm}^{1}(t)\|_{L^p(\R^k\times \mathbb{T}^{d-k})}\leq& C\|e^{\mu t\triangle}\rho_{\pm}^{00}\|_{L^p(\R^k\times \mathbb{T}^{d-k})}\\
\leq& C\|e^{\mu t\triangle}\rho_{\pm}(0)\|_{L^p(\R^k\times \mathbb{T}^{d-k})}\|N_1\|_{L^1(\R^d)}
\end{align*}
 for $2\leq p\leq \infty.$ Thus, we obtain
 \begin{align*}
 \|z_{\pm}(t)\|_{W^{1,\infty}(\R^k\times \mathbb{T}^{d-k})}\leq& C\|\rho^1_\pm(t)\|_{L^{\infty}(\R^k\times \mathbb{T}^{d-k})}\leq C\|e^{\mu t\triangle}\rho_{\pm}(0)\|_{L^{\infty}(\R^k\times \mathbb{T}^{d-k})}\\
 \leq& C(1+\mu t)^{-\frac{k}{4}}\|z_{\pm}(0)\|_{H^N(\R^k\times \mathbb{T}^{d-k})}.
 \end{align*}
 If $\rho_{\pm}(0)\in L^1(\R^d)$, we similarly have
\beno
&&\|z_{\pm}(t)\|_{H^N(\R^k\times \mathbb{T}^{d-k})}\leq C(1+\mu t)^{-\frac{k}{4}} ,\\
&&\|z_{\pm}(t)\|_{W^{1,\infty}(\R^k\times \mathbb{T}^{d-k})}\leq C(1+\mu t)^{-\frac{k}{2}}.
\eeno

\section*{Acknowledgements}

This work was supported by NSF of China under Grant 11425103.


 \end{document}